\documentclass[10pt, twocolumn]{IEEEtran}

\usepackage{amsmath}
\usepackage{amsthm}
\usepackage{amssymb}
\usepackage{verbatim}
\usepackage{algorithm}
\usepackage{algorithmic}
\usepackage{graphicx}
\usepackage{epstopdf}
\usepackage{paralist}
\usepackage{color}
\usepackage[square,numbers,sort&compress]{natbib} 

\ifCLASSOPTIONcompsoc
\usepackage[tight,normalsize,sf,SF]{subfigure}
\else
\usepackage[tight,footnotesize]{subfigure}
\fi

\theoremstyle{plain}
\newtheorem{proposition}{Proposition}
\newtheorem{lemma}{Lemma}
\newtheorem{assumption}{Assumption}
\newtheorem{theorem}{Theorem}

\newtheorem{definition}{Definition}

\theoremstyle{definition}

\floatname{algorithm}{Algorithm}
\newtheorem*{remark}{Remark}
\newcommand{\argmax}{\operatornamewithlimits{argmax}}

\allowdisplaybreaks[4]

\begin{document}
%
\title{Optimality of Myopic Policy for Restless Multiarmed Bandit with Imperfect Observation}
%
%
%
%

\author{
\IEEEauthorblockN{Kehao~Wang}
\IEEEauthorblockA{Wuhan University of Technology\\
 Hubei, P.R.C.\\
Email: kehao.wang@whut.edu.cn}
}

\IEEEcompsoctitleabstractindextext{%
\begin{abstract}
We consider the scheduling problem concerning $N$ projects. Each project evolves as a multi-state Markov process. At each time instant, one project is scheduled to work, and some reward depending on the state of the chosen project is obtained. The objective is to design a scheduling policy that maximizes the expected accumulated discounted reward over a finite or infinite horizon. The considered problem can be cast into a restless multi-armed bandit (RMAB) problem that is of fundamental importance in decision theory. It is well-known that solving the RMAB problem is PSPACE-hard, with the optimal policy usually intractable due to the exponential computation complexity. A natural alternative is to consider the easily implementable myopic policy that maximizes the immediate reward. In this paper, we perform an analytical study on the considered RMAB problem, and establish a set of closed-form conditions to guarantee the optimality of the myopic policy.
\end{abstract}

}

\maketitle

\IEEEdisplaynotcompsoctitleabstractindextext

%
\IEEEpeerreviewmaketitle

\section{Introduction}
\label{section:introduction}

Consider a scheduling system composed of $N$ independent projects each of which is models as a $X$-state Markov chain with known matrix of transition probabilities. At each time period one project is scheduled to work and a reward depending on the states of the worked project is obtained. The objective is to design a scheduling policy that maximizing the expected accumulated discounted reward (respectively, the expected accumulated reward) collected over a finite (respectively, infinite) time horizon. Mathematically, the considered channel access problem can be cast into the restless multi-armed bandit (RMAB) problem of fundamental importance in decision theory~\cite{Whittle88}.
RMAB problems arise in many areas, such as wired and wireless communication systems, manufacturing systems, economic systems, statistics, biomedical engineering, and information systems etc.~\cite{Whittle88,Gittins2011}. However, the RMAB problem is proved to be PSPACE-Hard~\cite{Papadimitriou99}.

The considered problem can also be formulated as a multi-state Partially Observed Markov Decision Process (POMDP)~\cite{Zhao07}. The challenges of multistate POMDPs are twofold: First, the probability vector is not completely ordered in the probability space, making the structural analysis substantially more difficult; Second, multistate POMDPs tend to encounter the ``curse of dimensionality", which is further complicated by the uncountably infinite probability space. Hence, numerical methods are adopted popularly. However, the numerical approach does not provide any meaningful insight into optimal policy. Moreover, this numerical approach has huge computational complexity. For the two reasons, we study some instances of the generic RMAB in which the optimal policy has a simple structure. Specially, we develop some sufficient conditions to guarantee the optimality of the myopic policy; that is, the optimal policy is to access the best channels each time in the sense of monotonic likelihood ratio order.

In the classic RMAB problem, a player chooses $M$ out of $N$ arms, each evolving as a Markov chain, to activate each time,  and receives a reward determined by the states of the activated arms. The objective is to maximize the long-run reward over an infinite horizon by choosing which $N$ arms to activate each time. If only the activated arms change their states, the problem is degenerated to the multi-armed bandit (MAB) problem \cite{Gittins74}.
The MAB problem is solved by Gittins by showing that the optimal policy has an index structure \cite{Gittins74, Gittins79}.

There exist two major thrusts in the research of the RMAB problem.
Since the optimality of myopic policy is not generally guaranteed, the first research thrust is to analyze the performance difference between optimal policy and approximation policy~\cite{Guha07,Guha09,Bertsimas00}. Specifically, a simple myopic policy, also called greedy policy, is developed in~\cite{Guha07} which yields a factor $2$ approximation of the optimal policy for a subclass of scenarios referred to as \emph{Monotone MAB}. The second thrust is to establish sufficient conditions to guarantee the optimality of the myopic policy in some specific instances of restless bandit scenarios, particularly in the context of opportunistic communications~\cite{Qzhao08,Sahmand09,Ahmad09,Kliu10,Wang11TPS,Fabio11icc,Wang13TVT,Wang13JSTSP}.

For the case of \emph{two-state}, Zhao~\emph{et al.}~\cite{Qzhao08} established the structure of the myopic policy, and partly obtained the optimality for the case of i.i.d. channels. Then Ahmad and Liu~\emph{et al.}~\cite{Ahmad09b} derived the optimality of the myopic sensing policy for the positively correlated i.i.d. channels for accessing one channel (i.e., $k=1$) each time, and further extended the optimality to access multiple i.i.d. channels ($k>1$)~\cite{Ahmad09}.
From another point, in~\cite{Wang11TPS}, we extended i.i.d. channels~\cite{Ahmad09b} to non i.i.d. ones, and focused on a class of so-called \emph{regular} functions, and derived closed-form sufficient conditions to guarantee the optimality of myopic sensing policy.
The authors~\cite{Fabio11icc} studied the myopic channel probing policy for the similar scenario proposed, but only established its optimality in the particular case of probing one channel ($M=1$) each time. In our previous work~\cite{Wang13TVT}, we established the optimality of myopic policy for the case of probing $N-1$ of $N$ channels each time and analyzed the performance of the myopic probing policy by domination theory, and further in~\cite{Wang13JSTSP} studied the generic case of arbitrary $M$ and derived more strong conditions on the optimality by dropping one of the non-trivial conditions of~\cite{Fabio11icc}.

For the complicated case of \emph{multi-state}, the authors in~\cite{ouyang14} established the sufficient conditions for the optimality of myopic sensing policy in multi-state homogeneous channels with a set of non-trivial assumptions.

\subsection{Contribution of the Paper}
The main results of this paper are the optimality conditions for expected accumulated discounted reward in Theorem~\ref{theorem:optimal_condition_heter_pos_case} and Theorem~\ref{theorem:optimal_condition_heter_neg_case} for imperfect observation, which makes it different from the most relevant paper~\cite{ouyang14} with perfect observation. The major difficulties encountered in optimizing the rewards in multi-state channel are: 1) how to obtain a non-trivial upper bound for multiple different stochastic matrices under multivariate reward (corresponding to multi-state) case; 2) how to determine the stochastic order of belief vectors; 3) identify the number of branches in the decision tree determined by a specific policy corresponding to the auxiliary value function defined in this paper.
These issues are resolved by 1) assuming that each transmission matrix has a non-trivial eigenvalue with $X-1$ times, under which the first-order stochastic dominance is preserved and meanwhile, the upper bound of each matrix is characterized by the eigenvalue; 2) assuming that there exists a determined stochastic dominance order of transmission matrices at any time instance; 3) considering the performance difference of two specific policies which differ in only one element of belief vectors; that is, the two policies have the form of \textit{difference}, mathematically. Further, we obtain the number of branches needed to be fix their bounds.

In this paper, we considered the problem of indirect observation of project states which makes our scheduling problem is different from~\cite{ouyang14} to a large extent.
In particular, the contributions of this paper include:
\begin{itemize}
  \item The structure of the myopic policy is shown to be a simple queue determined by the information states of projects provided that certain conditions are satisfied for the transition matrix of multi-state projects.
   \item We establish a set of conditions under which the myopic policy is proved to be optimal.
   \item Our derivation demonstrates the advantage of branch-and-bound and the directed comparison based optimization approach. The results of this paper are a generic contribution to the state of the art of the theory of restless bandit problems, although the structure of the optimal policy of generic restless bandit is not known.
\end{itemize}

\subsection{Organization}
The rest of the paper is organized as follows. In Section~\ref{section:Problem_formulation}, we present the system model and the formulation of the optimization problem. In Section~\ref{sec:opt}, we construct a set of conditions to guarantee the optimality of myopic policy by deriving some properties of transmission matrix and some bounds of serval pairs of policies. In Section~\ref{sec:ext}, the optimality results are extended to two different cases. Finally, we conclude in Section~\ref{section:conclusion}.

\section{Problem Formulation}
\label{section:Problem_formulation}

Consider $N$ independent projects $n=1,\cdots,N$. Assume each project $n$ has a finite number, $X$, of states, denoted as $\mathcal{X}$. Let $s_t^{(n)}$ denote the state of project $n$ at discrete time $t=1,2,\cdots$. At each time instant $t$, only one of these projects can be worked on. If project $n$ is worked on at time $t$, an instantaneous reward $\beta^t R(s_t^{(n)},n)$ is accrued ($ R(s_t^{(n)},n)$ is assumed finite). Here, $0\leq \beta \leq 1$ denotes the discount factor; the state $s_t^{(n)}$ evolves according to an $X$-state homogeneous Markov chain with transition probability matrix
$A=(a_{ij})_{i,j\in \mathcal{X}}$, where, $$a_{ij}=P(s^{(n)}_{t+1}=j | s^{(n)}_{t}=i) \text{  if project $n$ is worked on at $t$}.$$
All projects are initialized with $s^{(n)}_0 \sim x^{(n)}_0$, where $x^{(n)}_0$ are specified initial distributions for $n=1,\cdots, N$.

The state of the active project $n$ is indirectly observed via noisy measurements (observations) $y^{(n)}_{t+1}$ of the active project state $s^{(n)}_{t+1}$. Assume that these observations $y^{(n)}_{t+1}$ belong to a finite set $\mathcal{Y}$ indexed by $m=1,\cdots,\mathcal{Y}$. Let $B=(b_{im})_{i\in \mathcal{X}, j\in \mathcal{Y}}$ denote the observation probability matrix of the HMM, where each element $b_{im}\triangleq P(y^{(n)}_{t+1}=m | y^{(n)}_{t}=i, u_t=n)$.

Let $u_t\in \{1,\cdots,N\}$ denote which project is worked on at time $t$. Consequently, $s^{(u_t)}_{t+1}$ denotes the  state of the active project at time $t+1$. Denote the observation history at time $t$ as $Y_t=(y^{(u_0)}_1, \cdots, y^{(u_{t-1})}_t)$ and let $U_t=(u_0,\cdots,u_t)$. Then the project at time $t+1$ is chosen according to $u_{t+1}=\mu(Y_{t+1},U_t)$, where the policy denoted as $\mu$ belongs to the class of stationary policies $\mathcal{U}$. The total expected discounted reward over an infinite-time horizon is given by
\begin{equation}\label{ob:obj}
    J_{\mu}=\mathbb{E}\Big[\sum^{\infty}_{t=0}\beta^t R(s^{(u_t)}_{t},u_t)\Big], ~~~~u_t=\mu(Y_t,U_{t-1}),
\end{equation}
where $\mathbb{E}$ denotes mathematical expectation. The aim is to determine the optimal stationary policy $\mu^{\ast}=\argmax_{\mu\in \mathcal{U}} J_{\mu}$, which yields the maximum rewards in~\eqref{ob:obj}.

\subsection{Information state}
The above partially observed multiarmed bandit problem can be re-expressed as a fully observed multiarmed bandit in terms of the information state. For each project $n$, denoted by $x^{(n)}_t$ the information state at time $t$ (Bayesian posterior distribution of $s^{(n)}_t$) as $x^{(n)}_t=(x^{(n)}_{t}(i))~~i=1,\cdots,X$, where $x^{(n)}_{t}(i)\triangleq P(s^{(n)}_{t}=i|Y_t,U_{t-1})$.
The HMM multiarmed bandit problem can be viewed as the following scheduling problem: Consider $N$ parallel HMM state estimation filters, one for each project. The project $n$ is active, an observation $y^{(n)}_{t+1}$ is obtained and the information state $x^{(n)}_{t+1}$ is computed recursively by the HMM state filter according to
\begin{equation*}
    x^{(n)}_{t+1}=T(x^{(n)}_t,y^{(n)}_{t+1}), \text{ if project $n$ is worked on at time $t$},
\end{equation*}
where
\begin{align}\label{eq:pr_tran_o}
    &T(x^{(n)},y^{(n)})\triangleq \frac{B(y^{(n)})A'x^{(n)}}{d(x^{(n)},y^{(n)})}, \\
    &d(x^{(n)},y^{(n)})\triangleq \mathbf{1}'_{X}B(y^{(n)})A'x^{(n)}\nonumber .
\end{align}
In~\eqref{eq:pr_tran_o}, if $y^{(n)}=m$, then $B(m)=diag[b_{1m},\cdots,b_{X m}]$ is the diagonal matrix formed by the $m$th column of the observation matrix $B$, $A_x$ is the $x$th row of the matrix $A$, and $\mathbf{1}_{X}$ is an $X$-dimensional column vector of ones.

The state estimation of the other $N-1$ projects is according to
\begin{equation}\label{eq:pr_tran}
     x^{(n)}_{t+1}=A'x^{(n)}_t,
\end{equation}
if project $l$ is not worked on at time $t$, $l\in\{1,\cdots,N\},~~l\neq n$.

Let $\Pi(X)$ denote the state space of information states $x^{(n)},~n\in\{1,2,\cdots,N\}$, which is a $X-1$-dimensional simplex:
\begin{equation*}
   \Pi(X)=\Big\{ x \in\mathbb{R}^{X}: \mathbf{1}'_{X} x=1, 0\leq  x(i) \leq1 \text{ for all } i\in \mathcal{X} \Big\}.
\end{equation*}

The process $x^{(n)}_t,~n=1,\cdots,N$, qualifies as an information state since choosing $u_{t+1}=\mu(Y_{t+1},U_t)$ is equivalent to choosing $u_{t+1}=\mu(x^{(1)}_{t+1},\cdots,x^{(N)}_{t+1})$. Using the smoothing property of conditional expectations, the reward function \eqref{ob:obj} can be rewritten in terms of the information state as
\begin{equation*}
    J_{\mu}=\mathbb{E}\Big[\sum^{\infty}_{t=0}\beta^t R'(u_t)x^{(u_t)}_t\Big], ~~~~u_t=\mu(x^{(1)}_t,\cdots,x^{(N)}_t),
\end{equation*}
where $R'(u_t)$ denotes the $X$ dimensional reward column vector$[R(s^{(n)}_{t}=1,u_t),\cdots,R(s^{(n)}_{t}=X,u_t)]$. The aim is to compute the optimal policy $\argmax_{\mu\in \mathcal{U}} J_{\mu}$.

To get more insight on the structure of the optimization problem formulated in~\eqref{ob:dp}, we derive its dynamic programming formulation as follows:
\begin{align}\label{ob:dp}
\begin{cases}
V_{T}(x^{(1:N)}_T)=\max_{\substack{u_T}} \mathbb{E}\big[R'(u_T)x^{(u_T)}_T\big],  \\
V_{t}(x^{(1:N)}_t) = \max_{\substack{u_t}} \mathbb{E}\Big[ R'(u_t) x^{(u_t)}_t\\
\quad  + \beta  \sum_{m\in\mathcal{Y}} d(x^{(u_t)}_t,m) V_{t+1}(x^{(1:u_t-1)}_{t+1},x^{(u_t)}_{t+1,m},x^{(u_t+1:N)}_{t+1}) \Big],
\end{cases}
\end{align}
where, $x^{(i:j)}_t\triangleq \big(x^{(i)}_t,x^{(i+1)}_t,\cdots,x^{(j)}_t\big)$, and
\begin{equation}\label{ob:dp_pr}
\begin{cases}
x^{(u_t)}_{t+1,m} = T(x^{(u_t)}_t,m )\\
x^{(n)}_{t+1} = A'x^{(n)}_t,~~n \neq u_t.
\end{cases}
\end{equation}

\subsection{Myopic Policy}
Theoretically, the optimal policy can be obtained by solving the above dynamic programming. It is infeasible, however, due to the impact of the current action on the future reward, and in fact obtaining the optimal solution directly from the above recursive equations is computationally prohibitive. Hence, a natural alternative is to seek a simple myopic policy maximizing the immediate reward while ignoring the impact of the current action on the future reward, which is easy to compute and implement, formally defined as follows:
\begin{equation}
\hat{u}(t)=\argmax_{n} R'x^{(n)}_t .
\label{eq:myopic_policy}
\end{equation}


For the purpose of tractable analysis, we introduce some partial orders used in the following sections.

\begin{definition}[MLR ordering, ~\cite{Muller2002}]
Let $x_1$, $x_2\in \Pi(X)$ be any two belief vectors. Then $x_1$ is greater than $x_2$ with respect to the MLR ordering---denoted as  $x_1 ~{\geq}_r~ x_2$, if
\begin{equation*}
    x_{1}(i)x_{2}(j) \leq x_{2}(i)x_{1}(j), ~~~~i>j,~~i,j\in\{1,2,\cdots,X\}.
\end{equation*}
\end{definition}

\begin{definition}[first order stochastic dominance, ~\cite{Muller2002}]
Let $x_1$, $x_2\in \Pi(X)$, then $x_1$ first order stochastically dominates $x_2$---denoted as  $x_1~ {\geq}_s ~x_2$, if the following exists for $j=1,2,\cdots,X$,
\begin{equation*}
    \sum^X_{i=j}x_{1}(i) \geq \sum^X_{i=j}x_{2}(i).
\end{equation*}
\end{definition}

Some useful results~\cite{Muller2002} are stated here:
\begin{proposition}[\cite{Muller2002}]\label{prop:result}
Let $\mathbf{w}_1$, $\mathbf{w}_2\in \Pi(X)$, the following holds
\begin{enumerate}
  \item $\mathbf{w}_1 {\geq}_r\mathbf{w}_2$ implies $\mathbf{w}_1 {\geq}_s\mathbf{w}_2$.
  \item Let $\mathcal{V}$ denote the set of all $X$ dimensional vectors $v$ with nondecreasing components, i.e., $v_1\leq v_2\leq \cdots \leq v_X$. Then $\mathbf{w}_1 {\geq}_s\mathbf{w}_2$ iff for all $v\in \mathcal{V}$, $v' \mathbf{w}_1 \geq v' \mathbf{w}_2 $.
\end{enumerate}
\end{proposition}

\begin{definition}[Myopic Policy]
The myopic policy $\hat{u}:=(\hat{u}_0,\hat{u}_1,\cdots,\hat{u}_T)$ is the policy that selects the best project (in the sense of MLR) at each time. That is, if $x^{(\sigma_1)}_t {\geq}_s \cdots {\geq}_s x^{(\sigma_N)}_t $, then the myopic policy at $t$ is
\begin{align*}
   \hat{u}_t=\mu_t(x^{(1)}_t,\cdots,x^{(N)}_t)=\sigma_1.
\end{align*}
\end{definition}

\section{Optimality}
\label{sec:opt}

To analyze the performance of the myopic policy, we first introduce an auxiliary value function and then prove a critical feature of the auxiliary value function. Next, we give a simple assumption about transmission matrix, and show its special stochastic order. Finally, by deriving the bounds of different policies, we get some important bounds, which serves as the basis to prove the optimality of the myopic policy.

\subsection{Value Function and its Properties}
First, we define the auxiliary value function (AVF) as follows:
\begin{small}
\begin{equation}
\label{Eq:vf}
\begin{cases}
W^{\hat{u}}_{T}(x^{(1:N)}_T) = R'(u_T)x^{(u_T)}_T, \\
W^{\hat{u}}_{\tau}(x^{(1:N)}_{\tau}) = R'(\hat{u}_{\tau})x^{(\hat{u}_{\tau})}_T\\
\quad+\beta \underbrace{  \sum_{m\in\mathcal{Y}} d(x^{(\hat{u}_{\tau})}_{\tau},m ) W^{\hat{u}}_{\tau+1}(x^{(1:\hat{u}_{\tau}-1)}_{\tau+1},x^{(\hat{u}_{\tau})}_{\tau+1,m},x^{(\hat{u}_{\tau}+1:N)}_{\tau+1})}_{\digamma(x^{(1:N)}_{\tau},\hat{u}_{\tau})},~~t+1 \leq \tau \leq T\\
W^{u}_{t}(x^{(1:N)}_t) =  R'(u_t)x^{(u_t)}_t\\
\quad+\beta \underbrace{\sum_{m\in\mathcal{Y}} d(x^{(u_t)}_t,m) W^{\hat{u}}_{t+1}(x^{(1:u_t-1)}_{t+1},x^{(u_t)}_{t+1,m},x^{(u_t+1:N)}_{t+1})}_{\digamma(x^{(1:N)}_t,u_t)},
\end{cases}
\end{equation}
\end{small}

\begin{remark} AVF is the reward under the policy: at slot $t$, $u_t$ is adopted, while after $t$, myopic policy $\hat{u}_{\tau}$ ($t+1 \leq \tau \leq T$) is adopted.
\end{remark}

Let $e_i$ be an $X$-dimensional column vector with 1 in the $i$-th element and 0 in others, and $E$ be the $X\times X$ unit matrix.
\begin{lemma}
\label{lemma:decomposability_sym_ARQ}
$W^{u}_{t}(x^{(1:N)}_t)$ is decomposable for all $t=0,1, \cdots, T$, i.e.,
\begin{align*}
&W^{u}_{t}(x^{(1:n-1)}_t,x^{(n)}_t,x^{(n+1:N)}_t)\\
 &= \sum^X_{i=1}     x^{(n)}_t(i) W^{u}_{t}(x^{(1:n-1)}_t,e_i,x^{(n+1:N)}_t)\\
 &= \sum^X_{i=1}e'_i x^{(n)}_t    W^{u}_{t}(x^{(1:n-1)}_t,e_i,x^{(n+1:N)}_t)
\end{align*}
\end{lemma}
\begin{proof}
Please refer to Appendix A.
\end{proof}

\subsection{Assumptions}
We make the following assumptions/conditions.
\begin{assumption}\label{assp:1}
Assume that
\begin{itemize}
  \item[1)] $A_1 ~{\leq}_r ~A_2 {\leq}_r~ \cdots~ {\leq}_r~ A_X$.
  \item[2)] $B(1) ~{\leq}_r ~B(2) {\leq}_r~ \cdots~ {\leq}_r~ B(Y)$.
  \item[3)] There exists some $K$ ($2\leq K \leq Y$) such that
  \begin{align*}
    &T(A'e_1,K)~{\geq}_r ~(A')^2 e_1,\\
    &T(A'e_X,K-1)~{\leq}_r~ (A')^2 e_1.
  \end{align*}
  \item[4)] $A_1~{\leq}_r~x^{(1)}_0~{\leq}_r~x^{(2)}_0~{\leq}_r~\cdots~{\leq}_r~x^{(N)}_0~{\leq}_r~A_X$.
  \item[5)] $R' ( e_{i+1}- e_i)\geq R' Q' ( e_{i+1}- e_i)$ ($1\leq i \leq X-1$), where $A= V \Lambda V^{-1} $, $Q = V \Upsilon V^{-1} $,
 \begin{align*}
   & \Lambda= \left(
    \begin{array}{cccc}
      1 & 0 & \hdots & 0 \\
      0 & \lambda_2 & \hdots & 0 \\
      \vdots & \vdots & \ddots & \vdots\\
      0 & 0 & \hdots & \lambda_X \\
    \end{array}
  \right),\\
  &\Upsilon= \left(
    \begin{array}{cccc}
      1 & 0 & \hdots & 0 \\
      0 & \frac{\beta\lambda_2}{1-\beta\lambda_2} & \hdots & 0 \\
      \vdots & \vdots & \ddots & \vdots\\
      0 & 0 & \hdots & \frac{\beta\lambda_X}{1-\beta\lambda_X} \\
    \end{array}
  \right).
 \end{align*}
 \end{itemize}
\end{assumption}
\begin{remark}
Assumption~\ref{assp:1}.1 ensures that the higher the quality of the channel's current state the higher is the likelihood that the next channel state will be of high quality.
Assumption~\ref{assp:1}.3 along with 1.1-1.2 ensure that the information states of all projects can be ordered at all times in the sense of stochastic order.
Assumption~\ref{assp:1}.4  states that initially the channels can be ordered in terms of their quality.
Assumption~\ref{assp:1}.5  states that the instantaneous rewards obtained at different states of the channel are sufficiently separated.
\end{remark}

\subsection{Properties}
Under Assumption~\ref{assp:1}.1-\ref{assp:1}.5, we have some important propositions concerning the structure of information state in the following, which are proved in Appendix~\ref{ap:Prop}.

\begin{proposition}\label{pro:A_inc}
Let $x_1,x_2\in \Pi(X)$ and $x_1 ~{\leq}_r~x_2$, then $(A_1)'~{\leq}_r~A' x_1 ~{\leq}_r~A' x_2~{\leq}_r~(A_X)'$.
\end{proposition}

Proposition~\ref{pro:A_inc} states that if at any time $t$ the information states of two channels are stochastically ordered and none of these channels is chosen at $t$, then the same stochastic order between the information states at time $t+1$ is maintained.

\begin{proposition}\label{pro:T_inc}
Let $x_1,x_2\in \Pi(X)$ and $(A_1)'~{\leq}_r~x_1 ~{\leq}_r~x_2~{\leq}_r~(A_X)'$, then  $T(x_1,K)~{\leq}_r~T(x_2,K)$.
\end{proposition}

Proposition~\ref{pro:T_inc} states the increasing monotonicity of updating rule with information state for scheduled project.

\begin{proposition}\label{prop:O_inc}
Let $x\in \Pi(X)$ and $A_1~{\leq}_r~x ~{\leq}_r~A_X$, then  $T(x,k)~{\leq}_r~T(x,m)$ for any $1\leq k\leq m \leq Y$.
\end{proposition}

Proposition~\ref{prop:O_inc} states the increasing monotonicity of updating rule with the increasing number of observation state for scheduled project.

\begin{proposition}\label{prop:infsta_sepe}
Under Assumption~\ref{assp:1}, we have either $x^{(l)}_t ~{\leq}_s~ x^{(n)}_t $ or $x^{(n)}_t ~{\leq}_s~ x^{(l)}_t $ for all $l,n \in \{1,2,\cdots, N\}$ for all $t$.
\end{proposition}

Proposition~\ref{prop:infsta_sepe} states that under Assumption~\ref{assp:1}, the information states of all projects can be ordered stochastically at all times.

Now we give an important structural property on transition matrix in the following proposition.
\begin{proposition}
\label{prop:Tmatrix}
Suppose that transition matrix $A$ has $X$ eigenvalues $\lambda_1\geq \lambda_2\geq \cdots \geq \lambda_X$ and the corresponding orthogonal eigenvectors are $V_1,V_2,\cdots, V_X$. If $x_1,x_2 \in \Pi(X)$, then we have
\begin{itemize}
  \item $\lambda_1=1$ and $V_1=\frac{1}{\sqrt{X}} \mathbf{1}_X$;
  \item for any $\lambda$,
\begin{equation}\label{eq:matrix}
    \Lambda_1V'_1 (x_1-x_2)=\Lambda_2V'_1 (x_1-x_2),
\end{equation}
where
 \begin{align*}
 &\Lambda_1=
  \left(
    \begin{array}{cccc}
      \lambda_1 & 0 & \hdots & 0 \\
      0 & \lambda_2 & \hdots & 0 \\
     \vdots & \vdots &\ddots & \vdots \\
      0 & 0& \hdots & \lambda_4 \\
    \end{array}
  \right), \\
 &\Lambda_2=
  \left(
    \begin{array}{cccc}
      \lambda & 0 & \hdots & 0 \\
      0 & \lambda_2 & \hdots & 0 \\
     \vdots & \vdots &\ddots & \vdots \\
      0 & 0& \hdots & \lambda_4 \\
    \end{array}
  \right).
 \end{align*}
\end{itemize}
\end{proposition}

Proposition~\ref{prop:Tmatrix} states that 1) for any transition matrix, the largest eigenvalue is 1, named as \emph{trivial eigenvalue}, and its corresponding eigenvector is $\frac{1}{\sqrt{X}} \mathbf{1}_X$, named as \emph{trivial eigenvector}; 2) for any two information states, $x_1,x_2\in \Pi(X)$, one special equation holds where the largest eigenvalue 1 can be replacing by any value.

\begin{proposition}\label{prop:eig_sum}
 Given $x_1, x_2\in \Pi(X)$, we have $$\displaystyle R' \sum_{i=1}^{\infty}(\beta A')^i (x_1-x_2) = R' Q' (x_1-x_2) = R' (V \Upsilon V^{-1})' (x_1-x_2).$$
\end{proposition}

Proposition~\ref{prop:eig_sum} states that the accumulated reward difference between two different state information vectors can be simply written as a matrix form.

\begin{proposition}\label{prop:R_diff}
$R' ( e_{i}- e_j)\geq R' Q' ( e_{i}- e_j)$ ($1\leq j < i \leq X$).
\end{proposition}

\subsection{Analysis of Optimality}
We first give some bounds of performance difference on serval pairs of policies, and then derive the main theorem on the optimality of myopic policy.
\begin{lemma}
\label{lemma:bound_p}
Under Assumption~\ref{assp:1}, $\mathbf{x}^l_t=(x^{(-l)}_t, x^{(l)}_t)$, $\mathbf{\check{x}}^l_t=(x^{(-l)}_t, \check{x}^{(l)}_t)$, $x^{(l)}_t ~{\leq}_r~ \check{x}^{(l)}_t$,
we have for $1\le t\le T$
\begin{itemize}
 \item[(C1)] if $u'_t=u_t=l$,
\begin{align*}
    R'(\check{x}^{(l)}_t-x^{(l)}_t)
    &\le W^{u'}_t(\mathbf{\check{x}}^l_t)-W^{u}_t(\mathbf{x}^l_t) \\
    &\le \sum_{i=0}^{T-t}\beta^i R'(A')^i(\check{x}^{(l)}_t-x^{(l)}_t);
\end{align*}
  \item[(C2)] if $u'_t\neq l$, $u_t\neq l$, and $u'_t=u_t$,
\begin{align*}
    0 \le  W^{u'}_t(\mathbf{\check{x}}^l_t)-W^{u}_t(\mathbf{x}^l_t) \le \sum_{i=1}^{T-t}\beta^i R'(A')^i(\check{x}^{(l)}_t-x^{(l)}_t);
\end{align*}
  \item[(C3)]  if  $u'_t = l$ and $u_t\neq l$,
\begin{align*}
    0 \le  W^{u'}_t(\mathbf{\check{x}}^l_t)-W^{u}_t(\mathbf{x}^l_t) \le \sum_{i=0}^{T-t}\beta^i R'(A')^i(\check{x}^{(l)}_t-x^{(l)}_t).
\end{align*}
\end{itemize}
\end{lemma}
\begin{proof}
Please refer to Appendix~\ref{ap:bound_p}.
\end{proof}

\begin{remark}
We would like to emphasize on what conditions the bounds of Lemma~\ref{lemma:bound_p} are achieved. For (C1), the lower bound is achieved when project $l$ is scheduled at slot $t$ but never scheduled after $t$; the upper bound is achieved when $l$ is scheduled from $t$ to $T$. For (C2), the lower bound is achieved when project $l$ is never scheduled from $t$; the upper bound is achieved when $l$ is scheduled from $t+1$ to $T$. For (C3), the lower bound is achieved when project $l$ is never scheduled from $t$; the upper bound is achieved when $l$ is scheduled from $t$ to $T$.
\end{remark}

\begin{lemma}\label{lemma:exchange}
Under Assumption~\ref{assp:1}, we have then $W^{l}_t(x^{(1:N)}_t) > W^{n}_t(x^{(1:N)}_t)$ if $x^{(l)}_t ~{>}_r~ x^{(n)}_t$.
\end{lemma}

\begin{proof}
By Lemma~\ref{lemma:bound_p}, we have
\begin{align*}
&W^{l}_t(x^{(1:N)}_t)- W^{n}_t(x^{(1:N)}_t) \\
=&[W^{l}_t(x^{(-l)}_t,x^{(l)}_t)-W^{l}_t(x^{(-l)}_t,x^{(n)}_t)]\\
&-[W^{l}_t(x^{(-l)}_t,x^{(n)}_t)-W^{n}_t(x^{(-n)}_t,x^{(n)}_t)] \\
=&[W^{l}_t(x^{(-l)}_t,x^{(l)}_t)-W^{l}_t(x^{(-l)}_t,x^{(n)}_t)]\\
&-[W^{n}_t(x^{(-l)}_t,x^{(n)}_t)-W^{n}_t(x^{(-n)}_t,x^{(n)}_t)] \\
\geq&  R'(\check{x}^{(l)}_t-x^{(l)}_t) -\sum_{i=1}^{T-t}\beta^i R'(A')^i(\check{x}^{(l)}_t-x^{(l)}_t)\\
=& R' \Big(E-\sum_{i=1}^{T-t}(\beta A')^i \Big)(\check{x}^{(l)}_t-x^{(l)}_t) \\
\geq & R' \Big(E-\sum_{i=1}^{\infty}(\beta A')^i \Big)(\check{x}^{(l)}_t-x^{(l)}_t) \\
\overset{(a)}= & R' \Big(E-V \Upsilon V^{-1}\Big)(\check{x}^{(l)}_t-x^{(l)}_t) \\
=& R'(E- Q') \sum^X_{j=2}\\
&\Big[\sum^X_{i=j}(\check{x}^{(l)}_t(i)-x^{(l)}_t(i))(e_j-e_{j-1})+ x^{(l)}_t(j)(e_j-e_{1}) \Big] \\
=&  \sum^X_{j=2}\Big[\sum^X_{i=j}(\check{x}^{(l)}_t(i)-x^{(l)}_t(i))R'(E-Q')(e_j-e_{j-1})\\
&+ x^{(l)}_t(j)R'(E- Q')(e_j-e_{1}) \Big] \\
=&  \sum^X_{j=2}\Big[\sum^X_{i=j}(\check{x}^{(l)}_t(i)-x^{(l)}_t(i))[R'(e_j-e_{j-1})-R'Q'(e_j-e_{j-1})]\\
&+ x^{(l)}_t(j)[R'(e_j-e_{1})-R'Q'(e_j-e_{1})] \Big] \\
\overset{(b)}\geq &0,
\end{align*}
where, the equality (a) is from Proposition~\ref{prop:eig_sum}, and the inequality (b) is from Proposition~\ref{prop:R_diff}, and $\sum^X_{i=j}(\check{x}^{(l)}_t(i)-x^{(l)}_t(i))\geq 0$ is due to $\check{x}^{(l)}_t~{\geq}_s~x^{(l)}_t$ from Proposition~\ref{prop:result}.
\end{proof}
\begin{remark}
Lemma~\ref{lemma:exchange} states that scheduling the project with better information state would bring more reward.
\end{remark}

Based on Lemma~\ref{lemma:exchange}, we have the following theorem which states the optimal condition of the myopic policy.
\begin{theorem}
\label{theorem:optimal_condition_heter_pos_case}
Under Assumption~\ref{assp:1}, the myopic  policy is optimal.
\end{theorem}
\begin{proof}
When $T\nrightarrow \infty$, we prove the theorem by backward induction. The theorem holds trivially for $T$. Assume that it holds for $T-1, \cdots, t+1$, i.e., the optimal accessing policy is to access the best channels (in the sense of stochastic dominance in terms of ) from time slot $t+1$ to $T$. We now show that it holds for $t$.
Suppose, by contradiction, that given $x\triangleq \{x^{(i_1)}, \cdots, x^{(i_N)}\}$ and $x^{(1)}~{>}_s \cdots ~{>}_s ~x^{(N)}$, the optimal policy is to choose the best from time slot $t+1$ to $T$, and thus, at slot $t$, to choose $\mu_t= i_1 \neq 1 = \hat{\mu}_t$, given that the latter, $\hat{\mu}_t$, is to choose the best project in the sense of stochastic order at slot $t$. There must exist $i_n$ at slot $t$ such that $x^{(i_n)}~{>}_s ~ x^{(i_1)}$. It then follows from Lemma~\ref{lemma:exchange} that $W^{i_n}_t(x^{(1:N)}_t) > W_t^{i_1}(x^{(1:N)}_t)$,
which contradicts with the assumption that the latter is the optimal policy. This contradiction completes our proof for $T$.
When $T\rightarrow \infty$, the proof is finished.
\end{proof}

\subsection{Discussion}
\subsubsection{Comparison}
In~\cite{ouyang14}, the authors considered the problem of scheduling multiple channels with direct or perfect observation, and then the method is based on the information states of all channels in the sense of first order stochastic dominance order; that is, the critical property is to keep the information states completely ordered or separated in the sense of first order stochastic dominance order. However, in the case of indirect or imperfect observation, an observation matrix is introduced to replace the unit matrix $E$ for the direct observation considered in~\cite{ouyang14}. Hence, the stochastic dominance order is not sufficient to characterize the order of information states, and then the monotonic likelihood ratio order, a kind of more stronger stochastic order, is used to describe the order structure of information states.

The Assumption 1.5 is different from the Assumption (A4) of~\cite{ouyang14}.

\subsubsection{Bounds}
The bounds in (C1)-(C3) are not enough tight to drop the non-trivial Assumption~\ref{assp:1}.5. Actually, we conjecture the optimality of myopic policy is kept even without the Assumption~\ref{assp:1}.5. However, due to the constraint of the method adopted in this paper, we cannot obtain better bounds to drop the non-trivial Assumption~\ref{assp:1}.5. Therefore, one of further directions is to obtain the optimality of myopic policy without Assumption~\ref{assp:1}.5 by some new methods.

%

\section{Optimality Extension}
\label{sec:ext}

In this section, we first extend the obtained optimality results to the case in which the transition matrix is totally negative order, as a complementary to the totally positive order discussed in the previous section, which means that those relative propositions are stated here by replacing increasing monotonicity with deceasing monotonicity. Second, we extend the optimality to the case of scheduling multiple projects simultaneously.

\subsection{Assumptions}
Some important assumptions are stated in the following.
\begin{assumption}\label{assp:2}
Assume that
\begin{itemize}
  \item[1)] $A_1 ~{\geq}_r ~A_2 {\geq}_r~ \cdots~ {\geq}_r~ A_X$.
  \item[2)] $B(1) ~{\leq}_r ~B(2) {\leq}_r~ \cdots~ {\leq}_r~ B(Y)$.
  \item[3)] There exists some $K$ ($2\leq K \leq Y$) such that
   \begin{align*}
    &T(A'e_X,K)~{\leq}_r ~(A')^2 e_X,\\
    &T(A'e_1,K-1)~{\geq}_r~ (A')^2 e_X.
   \end{align*}
  \item[4)] $A_1~{\geq}_r~x^{(1)}_0~{\geq}_r~x^{(2)}_0~{\geq}_r~\cdots~{\geq}_r~x^{(N)}_0~{\geq}_r~A_X$.
  \item[5)] $R' ( e_{i+1}- e_i)\geq R' Q' ( e_{i+1}- e_i)$ ($1\leq i \leq X-1$), where $A= V \Lambda V^{-1} $, $Q = V \Upsilon V^{-1} $.
 \end{itemize}
\end{assumption}
\begin{remark}
Assumption~\ref{assp:2} differs from Assumption~\ref{assp:1} in three aspects, i.e., \ref{assp:2}.1,~\ref{assp:2}.3,~\ref{assp:2}.4, which reflects the inverse TP2 order~\cite{Muller2002} in matrix $A$.
\end{remark}

\subsection{Optimality}
Under Assumption~\ref{assp:2}, we have the following propositions similar to Proposition~\ref{pro:A_inc}---Proposition~\ref{prop:infsta_sepe}.
\begin{proposition}\label{pro:A_dec}
Let $x_1,x_2\in \Pi(X)$ and $x_1 ~{\leq}_r~x_2$, then $(A_1)'~{\geq}_r~A' x_1 ~{\geq}_r~A' x_2~{\geq}_r~(A_X)'$.
\end{proposition}

\begin{proposition}\label{pro:T_dec}
Let $x_1,x_2\in \Pi(X)$ and $(A_1)'~{\geq}_r~x_1 ~{\geq}_r~x_2~{\geq}_r~(A_X)'$, then  $T(x_1,K)~{\leq}_r~T(x_2,K)$.
\end{proposition}

\begin{proposition}\label{prop:O_dec}
Let $x\in \Pi(X)$ and $(A_1)'~{\geq}_r~x ~{\geq}_r~(A_X)'$, then  $T(x,k)~{\geq}_r~T(x,m)$ for any $1\leq k\leq m \leq Y$.
\end{proposition}

\begin{proposition}
Under Assumption 2, we have either $x^{(l)}_t ~{\leq}_s~ x^{(n)}_t $ or $x^{(n)}_t ~{\leq}_s~ x^{(l)}_t $ for all $l,n \in \{1,2,\cdots, N\}$ for all $t$.
\end{proposition}

Following the similar derivation of Lemma~\ref{lemma:bound_p}, we have the following important bounds.
\begin{lemma}
\label{lemma:bound_r}
Under Assumption~\ref{assp:2}, $\mathbf{x}^l_t=(x^{(-l)}_t, x^{(l)}_t)$, $\mathbf{\check{x}}^l_t=(x^{(-l)}_t, \check{x}^{(l)}_t)$, $x^{(l)}_t ~{\leq}_r~ \check{x}^{(l)}_t$,
we have for $1\le t\le T$
\begin{itemize}
 \item[(D1)] if $u'_t=u_t=l$,
\begin{align*}
    &R'\Big( E-\sum_{i=1}^{\lceil\frac{ T-t}{2}\rceil} (\beta A')^{2i-1} \Big)(\check{x}^{(l)}_t-x^{(l)}_t)\\
    &\le W^{u'}_t(\mathbf{\check{x}}^l_t)-W^{u}_t(\mathbf{x}^l_t) \\
    &\le R'\Big( E+\sum_{i=1}^{\lfloor\frac{ T-t}{2}\rfloor} (\beta A')^{2i} \Big)(\check{x}^{(l)}_t-x^{(l)}_t);
\end{align*}
  \item[(D2)] if $u'_t\neq l$, $u_t\neq l$, and $u'_t=u_t$,
\begin{align*}
   & -R'\sum_{i=1}^{\lceil\frac{ T-t}{2}\rceil} (\beta A')^{2i-1} (\check{x}^{(l)}_t-x^{(l)}_t) \\
    &\le  W^{u'}_t(\mathbf{\check{x}}^l_t)-W^{u}_t(\mathbf{x}^l_t) \\
    &\le R'\sum_{i=1}^{\lfloor\frac{ T-t}{2}\rfloor} (\beta A')^{2i} (\check{x}^{(l)}_t-x^{(l)}_t);
\end{align*}
  \item[(D3)]  if  $u'_t = l$ and $u_t\neq l$,
\begin{align*}
   & -R'\sum_{i=1}^{\lceil\frac{ T-t}{2}\rceil} (\beta A')^{2i-1} (\check{x}^{(l)}_t-x^{(l)}_t)  \\
    &\le  W^{u'}_t(\mathbf{\check{x}}^l_t)-W^{u}_t(\mathbf{x}^l_t) \\
    &\le R'\Big( E+\sum_{i=1}^{\lfloor\frac{ T-t}{2}\rfloor} (\beta A')^{2i} \Big)(\check{x}^{(l)}_t-x^{(l)}_t).
\end{align*}
\end{itemize}
\end{lemma}
\begin{remark}(D1) achieves its lower bound when $l$ is chosen at slot $t,t+1,t+3,\cdots$, and achieves the upper bound when $l$ is chosen from $t,t+2,t+4,\cdots$.
(D2) achieves its lower bound when $l$ is chosen at slot $t+1,t+3,\cdots$, and upper bounds when $l$ is chosen at $t+2,t+4,\cdots$.
(D3) achieves its lower bound when $l$ is chosen at slot $t+1,t+3,\cdots$, and upper bounds when $l$ is chosen from $t,t+2,t+4,\cdots$.
\end{remark}

Based on Lemma~\ref{lemma:exchange} and~\ref{lemma:bound_r}, we have the following theorem.
\begin{theorem}
\label{theorem:optimal_condition_heter_neg_case}
Under Assumption~\ref{assp:2}, the myopic  policy is optimal.
\end{theorem}

\subsection{Extension of Scheduling Multiple Projects Simultaneously}
It is necessary to point out that the method adopted and the bounds obtained in this paper can be trivially extended to the case of scheduling multiple projects simultaneously. In this case, the bounds in Lemmas~\ref{lemma:bound_p} and~\ref{lemma:bound_r} still hold without modifying any assumptions. This is because scheduling multiple projects simultaneously can be easily regarded as scheduling multiple projects one by one at each slot, while those non-scheduled projects remain their states. Therefore, the optimality of scheduling one project at each slot guarantees the optimality of scheduling multiple projects simultaneously under Assumption~\ref{assp:1} or~\ref{assp:2}.

\section{Conclusion}
\label{section:conclusion}

In this paper, we have investigated the problem of scheduling multi-state projects. In general, the problem can be formulated as a partially observable Markov decision process or restless multi-armed bandit, which is proved to be Pspace-hard. In this paper, we have derived a set of closed form conditions to guarantee the optimality of the myopic policy (scheduling the best project) in the sense of monotonic likelihood ratio order. Due to the generic RMAB formulation of the problem, the derived results and the analysis methodology proposed in this paper can be applicable in a wide range of domains.

\appendices

\section{Proof of Lemma~\ref{lemma:decomposability_sym_ARQ}}
For Slot $T$, it trivially holds. Suppose it holds for $T-1,\cdots,t+2,t+1$, we prove it holds for slot $t$.

At slot $t$, we prove it by two cases in the following.

Case 1: $u_t = n$,
\begin{align}\label{eq:dec1}
&W^{u}_{t}(x^{(1:n-1)}_t,x^{(n)}_t,x^{(n+1:N)}_t) \nonumber \\
=&  R'(n)x^{(n)}_t+ \beta \sum_{m\in\mathcal{Y}} d(x^{(n)}_t,m) W^{\hat{u}}_{t+1}(x^{(1:n-1)}_{t+1},x^{(n)}_{t+1,m},x^{(n+1:N)}_{t+1})\nonumber \\
\overset{(a)}=&  R'(n)x^{(n)}_t \nonumber \\
&+ \beta \sum_{m\in\mathcal{Y}} d(x^{(n)}_t,m) \sum^X_{j=1} e'_j x^{(n)}_{t+1,m} W^{\hat{u}}_{t+1}(x^{(1:n-1)}_{t+1},e_j,x^{(n+1:N)}_{t+1}),
\end{align}
where the equality (a) is due to the induction hypothesis.

\begin{align}\label{eq:dec2}
    &\sum^X_{i=1} x^{(n)}_t(i) W^{u}_{t}(x^{(1:n-1)}_t,e_i,x^{(n+1:N)}_t) \nonumber \\
   =&\sum^X_{i=1} x^{(n)}_t(i)\Big[ R'(n)x^{(n)}_t \nonumber \\ &
   + \beta \sum_{m\in\mathcal{Y}} d(e_i,m)W^{\hat{u}}_{t+1}(x^{(1:n-1)}_{t+1},T(e_i,m),x^{(n+1:N)}_{t+1})\Big] \nonumber \\
 \overset{(b)}   =&  R'(n)x^{(n)}_t+\beta \sum^X_{i=1} x^{(n)}_t(i) \sum_{m\in\mathcal{Y}} d(e_i,m) \nonumber \\
   &\qquad \times W^{\hat{u}}_{t+1}(x^{(1:n-1)}_{t+1},T(e_i,m),x^{(n+1:N)}_{t+1}) \nonumber\\
 \overset{(c)}   =&  R'(n)x^{(n)}_t+\beta \sum^X_{i=1} x^{(n)}_t(i) \sum_{m\in\mathcal{Y}} d(e_i,m) \sum^X_{j=1} e'_j T(e_i,m) \nonumber \\
   &\qquad \times W^{\hat{u}}_{t+1}(x^{(1:n-1)}_{t+1},e_j,x^{(n+1:N)}_{t+1}),
\end{align}
where, the equality (b) is from $\sum^X_{i=1} x^{(n)}_t(i)=1$, and equality (c) is due to induction hypothesis.

To prove the the lemma, it is sufficient to prove the following equation
\begin{align}\label{dec3}
   & \sum_{m\in\mathcal{Y}} d(x^{(n)}_t,m) \sum^X_{j=1} e'_j x^{(n)}_{t+1,m} \nonumber \\
   &=\sum^X_{i=1} x^{(n)}_t(i) \sum_{m\in\mathcal{Y}} d(e_i,m)\sum^X_{j=1} e'_j T(e_i,m).
\end{align}

Now, we have RHS and LHS of~\eqref{dec3} as follows
\begin{align}\label{dec4}
    &\sum_{m\in\mathcal{Y}} d(x^{(n)}_t,m) \sum^X_{j=1} e'_j x^{(n)}_{t+1,m} \nonumber \\
    &=\sum_{m\in\mathcal{Y}} d(x^{(n)}_t,m)\sum^X_{j=1} e'_j \frac{B(m)A'x^{(n)}_t}{d(x^{(n)}_t,m)} \nonumber\\
    &=\sum_{m\in\mathcal{Y}}\sum^X_{j=1} e'_j B(m)A'x^{(n)}_t.
\end{align}

\begin{align}\label{dec5}
&    \sum^X_{i=1} x^{(n)}_t(i) \sum_{m\in\mathcal{Y}} d(e_i,m)\sum^X_{j=1} e'_j T(e_i,m) \nonumber\\
    &= \sum^X_{i=1} x^{(n)}_t(i) \sum_{m\in\mathcal{Y}} d(e_i,m) \sum^X_{j=1} e'_j \frac{B(m)A'e_i}{d(e_i,m)} \nonumber\\
    &= \sum^X_{i=1} x^{(n)}_t(i) \sum_{m\in\mathcal{Y}} \sum^X_{j=1} e'_j B(m)A'e_i \nonumber\\
    &= \sum_{m\in\mathcal{Y}} \sum^X_{j=1} e'_j B(m)A' \sum^X_{i=1} x^{(n)}_t(i) e_i \nonumber\\
    &=\sum_{m\in\mathcal{Y}}\sum^X_{j=1} e'_j B(m)A'x^{(n)}_t.
\end{align}
Combing~\eqref{dec4} and~\eqref{dec5}, we have~\eqref{dec3}, and further, prove the lemma.

Case 2: $u_t \neq n$, without loss of generality, assuming $u_t \geq n+1$,
\begin{align}\label{eq:dec6}
&W^{u}_{t}(x^{(1:n-1)}_t,x^{(n)}_t,x^{(n+1:N)}_t) \nonumber \\
=&  R'(u_t)x^{(u_t)}_t+ \beta \sum_{m\in\mathcal{Y}} d(x^{(u_t)}_t,m) \nonumber \\ &W^{\hat{u}}_{t+1}(x^{(1:u_t-1)}_{t+1},x^{(u_t)}_{t+1,m},x^{(u_t+1:N)}_{t+1})\nonumber \\
\overset{(a)}=&  R'(u_t)x^{(u_t)}_t+ \beta \sum_{m\in\mathcal{Y}} d(x^{(u_t)}_t,m) \sum^X_{i=1} x^{(n)}_{t+1}(i) \nonumber\\
&W^{\hat{u}}_{t+1}(x^{(1:n-1)}_{t+1},e_i,x^{(n+1:u_t-1)}_{t+1},x^{(u_t)}_{t+1,m},x^{(u_t+1:N)}_{t+1}),
\end{align}
where, the equality (a) is due to the induction hypothesis.
\begin{align}\label{eq:dec7}
    &\sum^X_{i=1} x^{(n)}_t(i) W^{u}_{t}(x^{(1:n-1)}_t,e_i,x^{(n+1:N)}_t) \nonumber \\
   =&\sum^X_{i=1} x^{(n)}_t(i)\Big[ R'(u_t)x^{(u_t)}_t + \beta \sum_{m\in\mathcal{Y}} d(x^{(u_t)}_t,m) \nonumber \\
 &  W^{\hat{u}}_{t+1}(x^{(1:n-1)}_{t+1},e_i,x^{(n+1:u_t-1)}_{t+1},x^{(u_t)}_{t+1,m},x^{(u_t+1:N)}_{t+1})\Big] \nonumber \\
 \overset{(b)}   =&  R'(u_t)x^{(u_t)}_t+\beta \sum^X_{i=1} x^{(n)}_t(i) \sum_{m\in\mathcal{Y}} d(x^{(u_t)}_t,m) \nonumber \\
&   W^{\hat{u}}_{t+1}(x^{(1:n-1)}_{t+1},e_i,x^{(n+1:u_t-1)}_{t+1},x^{(u_t)}_{t+1,m},x^{(u_t+1:N)}_{t+1}),
\end{align}
where, the equality (b) is from $\sum^X_{i=1} x^{(n)}_t(i)=1$.

Combining \eqref{eq:dec6} and \eqref{eq:dec7}, we prove the lemma.

\section{Proof of Propositions~\ref{pro:A_inc}--\ref{prop:R_diff}}
\label{ap:Prop}

\subsection{Proof of Proposition~\ref{pro:A_inc}}
Suppose $i>j$, we have
\begin{align*}
    &(e'_i A' x_2)  \cdot (e'_j A' x_1)-  (e'_j A' x_2) \cdot (e'_i A' x_1) \\
    & = \sum^X_{k=1} a_{ki} x_2(k) \sum^X_{l=1} a_{lj} x_1(l) -  \sum^X_{k=1} a_{kj} x_2(k) \sum^X_{l=1} a_{li} x_1(l) \\
    & = \Big( \sum^X_{k=1}\sum^X_{l=1}  a_{ki} a_{lj} - \sum^X_{k=1} \sum^X_{l=1} a_{kj} a_{li} \Big) x_2(k) x_1(l) \\
    & = \Big( \sum^X_{l=1} \sum^X_{k=l}(  a_{ki} a_{lj} -  a_{li} a_{kj}) -\sum^X_{k=1} \sum^X_{l=k}(  a_{li} a_{kj}- a_{ki} a_{lj} )\Big) \\
    &\qquad \times x_2(k) x_1(l) \\
    & = \sum^X_{l=1} \sum^X_{k=l}(  a_{ki} a_{lj} -  a_{li} a_{kj})(x_2(k) x_1(l)-x_2(l) x_1(k))\geq 0,
\end{align*}
where, the last inequality is due to $A_k ~{\geq}_r~A_l~(k\geq l)$ and $x_2~{\geq}_r~x_1$.

Then
we have $(A_1)'=A'e_1~{\leq}_r~A' x_1 ~{\leq}_r~A' x_2~{\leq}_r~A'e_X=(A_X)'$ considering $e_1~{\leq}_r~x_1~{\leq}_r~x_2~{\leq}_r~e_X$.

\subsection{Proof of Proposition~\ref{pro:T_inc}}
According to Proposition~\ref{pro:A_inc}, we have $z_1=A'x_1~{\leq}_r~A'x_2=z_2$.
Suppose $i>j$, we have
\begin{align*}
   & (T(x_2,K))_i\cdot(T(x_1,K))_j- (T(x_2,K))_j\cdot(T(x_1,K))_i \\
   & =\frac{b_{iK}z_2(i)}{\sum^X_{x=1}b_{xK}z_2(x)}\cdot \frac{b_{jK}z_1(j)}{\sum^X_{x=1}b_{xK}z_1(x)} \\
   & \qquad -\frac{b_{jK}z_2(j)}{\sum^X_{x=1}b_{xK}z_2(x)} \cdot\frac{b_{iK}z_1(i)}{\sum^X_{x=1}b_{xK}z_1(x)} \\
   &= \frac{b_{iK} b_{jK} (z_2(i)z_1(j)-z_2(j)z_1(i))}{\sum^X_{x=1}b_{xK}z_2(x) \sum^X_{x=1}b_{xK}z_1(x)}\geq 0,
\end{align*}
where, $z_2(i)z_1(j)-z_2(j)z_1(i)\geq 0$ is from $z_1~{\leq}_r~z_2$.

\subsection{Proof of Proposition~\ref{prop:O_inc}}
Let $z=A'x$.
Suppose $i>j$, we have
\begin{align*}
    &(T(x,m))_i\cdot(T(x,k))_j-(T(x,m))_j\cdot(T(x,k))_i \\
    &=  \frac{b_{im}z(i)}{\sum^X_{l=1}b_{lm}z(l)}\cdot  \frac{b_{jk}z(j)}{\sum^X_{l=1}b_{lk}z(l)} \\
    &-  \frac{b_{jm}z(j)}{\sum^X_{l=1}b_{lm}z(l)}\cdot  \frac{b_{ik}z(i)}{\sum^X_{l=1}b_{lk}z(l)} \\
    &=\frac{(b_{im}b_{jk}-b_{jm}b_{ik})z(i)z(j)}{\sum^X_{l=1}b_{lm}z(l)\sum^X_{l=1}b_{lk}z(l)}\geq 0,
\end{align*}
where, $b_{im}b_{jk}-b_{jm}b_{ik}\geq 0$ is from $B(m)~{\geq}_r~B(k)$.

\subsection{Proof of Proposition~\ref{prop:infsta_sepe}}
Let $\phi(z)=\frac{B(K)z}{\mathbf{1}'_{X}B(K)z}$ where $z \in \Pi(X)$ and $(A_1)'~{\leq}_r~z~{\leq}_r~(A_X)'$. We first show that $\phi(z_1)-z_1~{\leq}_r~\phi(z_2)-z_2$ for $z_2~{\geq}_r~z_1$.
Suppose $i>j$, we have
\begin{align*}
   & (\phi(z_1)-z_1)_i\cdot(\phi(z_2)-z_2)_j - (\phi(z_1)-z_1)_j\cdot(\phi(z_2)-z_2)_i \\
   &=\Big(\frac{b_{iK}z_1(i)}{\sum^X_{l=1}b_{lK}z_1(l)}- z_1(i)\Big)\Big(\frac{b_{jK}z_2(j)}{\sum^X_{l=1}b_{lK}z_2(l)}- z_2(j)\Big) \\
   & -\Big(\frac{b_{jK}z_1(j)}{\sum^X_{l=1}b_{lK}z_1(l)}- z_1(j)\Big)\Big(\frac{b_{iK}z_2(i)}{\sum^X_{l=1}b_{lK}z_2(l)}- z_2(i)\Big) \\
   &=(z_1(i)z_2(j)-z_1(j)z_2(i))\\
   &\times\Big(\frac{b_{iK}}{\sum^X_{l=1}b_{lK}z_1(l)}- 1\Big)\Big(\frac{b_{jK}}{\sum^X_{l=1}b_{lK}z_2(l)}- 1\Big)\leq0,
\end{align*}
where, $z_1(i)z_2(j)-z_1(j)z_2(i)\leq 0$ is from $z_2~{\geq}_r~z_1$. Thus, we have $\phi(z_1)-z_1~{\leq}_r~\phi(z_2)-z_2$ for $z_2~{\geq}_r~z_1$.

According to Assumption~\ref{assp:1}.3, we have $\phi(z)-z =\frac{B(K)z}{\mathbf{1}'_{X}B(K)z}-z~{\geq}_r~ \mathbf{0}$ for any $z~{\geq}_r~ (A')^2 e_1$; that is, $T(x,K)-A'x~{\geq}_r ~\mathbf{0}$ for any $x~{\geq}_r~ A' e_1$. Combining Proposition~\ref{prop:O_inc}, we $T(x,k)-A'x~{\geq}_r ~0$ for $k\geq K$ and any $x~{\geq}_r~ A' e_1$.

According to Assumption~\ref{assp:1}.3 and Proposition~\ref{prop:O_inc}, we $T(x,k)~{\leq}_s ~A'e_1$ for $k\leq K-1$ and any $x~{\geq}_r~ A' e_1$.

Thus, we have the proposition.

\subsection{Proof of Proposition~\ref{prop:Tmatrix}}
(1) For the property of $\lambda_1=1$ and $V_1=\frac{1}{\sqrt{X}}\mathbf{1}_X$, it is easily verified, i.e.,
\begin{align*}
\frac{1}{\sqrt{X}} A \cdot \mathbf{1}_X  &
=\frac{1}{\sqrt{X}} \left(
    \begin{array}{c}
      A_1 \cdot \mathbf{1}_X  \\
      A_2 \cdot \mathbf{1}_X  \\
       \vdots  \\
      A_X \cdot \mathbf{1}_X
     \end{array}
\right)
=\frac{1}{\sqrt{X}} \left(
    \begin{array}{c}
      1 \\
      1 \\
      \vdots \\
      1
     \end{array}
\right)= \frac{1}{\sqrt{X}} \mathbf{1}_X .
\end{align*}

(2) For the property of replacing $\lambda_1$ with any value $\lambda$, we have the LHS of~\eqref{eq:matrix}
\begin{align}\label{eq:lam1}
 & \Lambda_1V' (x_1-x_2) \nonumber \\
  =&[\lambda_1V_1(x_1-x_2),~\lambda_2V_2(x_1-x_2),~\cdots,~\lambda_XV_X(x_1-x_2)]' \nonumber\\
  =&[\lambda_1\frac{1}{\sqrt{X}}\mathbf{1}_X(x_1-x_2),~\lambda_2V_2(x_1-x_2),~\cdots,~\lambda_XV_X(x_1-x_2)]' \nonumber\\
  =&[0~~\lambda_2V_2(x_1-x_2),~\cdots,~\lambda_XV_X(x_1-x_2)]'.
\end{align}
For the RHS of~\eqref{eq:matrix}, we have
\begin{align}\label{eq:lam2}
 & \Lambda_2V' (x_1-x_2) \nonumber \\
  =&[\lambda V_1(x_1-x_2),~\lambda_2V_2(x_1-x_2),~\cdots,~\lambda_XV_X(x_1-x_2)]' \nonumber\\
  =&[\lambda_1\frac{1}{\sqrt{X}}\mathbf{1}_X(x_1-x_2),~\lambda_2V_2(x_1-x_2),~\cdots,~\lambda_XV_X(x_1-x_2)]' \nonumber\\
  =&[0~~\lambda_2V_2(x_1-x_2),~\cdots,~\lambda_XV_X(x_1-x_2)]'.
\end{align}
By~\eqref{eq:lam1} and~\eqref{eq:lam2}, we prove the equation~\eqref{eq:matrix}.

\subsection{Proof of Proposition~\ref{prop:eig_sum}}
\begin{align*}
      R' \sum_{i=1}^{\infty}(\beta A')^i (x_1-x_2)
   &= R' \sum_{i=1}^{\infty}(\beta (V^{-1})'\Lambda V')^i (x_1-x_2) \\
   &\overset{(a)}= R' \sum_{i=1}^{\infty}(\beta (V^{-1})'\Lambda_2 V')^i (x_1-x_2) \\
   & = R'(V^{-1})' \sum_{i=1}^{\infty}( \beta\Lambda_2)^i V'(x_1-x_2) \\
   & = R'(V^{-1})' \Upsilon V' (x_1-x_2) \\
   & = R' (V \Upsilon V^{-1})' (x_1-x_2) \\
   & = R' Q' (x_1-x_2),
\end{align*}
where, the equality (a) is due to Proposition~\ref{prop:Tmatrix}.

\subsection{Proof of Proposition~\ref{prop:R_diff}}
 According to Assumption~\ref{assp:1}.5, we have $R' ( e_{j+1}- e_j)\geq R' Q' ( e_{j+1}- e_j)$ ($1\leq j \leq X-1$). Thus, we only need to prove $R' ( e_{i}- e_j)\geq R' Q' ( e_{i}- e_j)$ for any $i>j+1$.
\begin{align*}
 &  R' ( e_{i}- e_j) -R' Q' ( e_{i}- e_j) \\
 &= R' \sum^{i-1}_{k=j}( e_{k+1}- e_k) -R' Q' \sum^{i-1}_{k=j}( e_{k+1}- e_k) \\
 &= \sum^{i-1}_{k=j}\Big[R' ( e_{k+1}- e_k) -R' Q' ( e_{k+1}- e_k)\Big] \geq 0,
\end{align*}
where, the last inequality is from Assumption~\ref{assp:1}.5.

\section{Proof of Lemma~\ref{lemma:bound_p}}
\label{ap:bound_p}

We prove the lemma by backward induction.

For slot $T$, we have
\begin{enumerate}
  \item[1)] For $u'_T=u_T=l$, it holds that $W^{u'}_T(\mathbf{\check{x}}^l_T)-W^{u}_T(\mathbf{x}^l_T)=R'(\check{x}^{(l)}_T-x^{(l)}_T)$;
  \item[2)] For $u'_T\neq l$, $u_T\neq l$ and $u'_T=u_T$, it holds that $W^{u'}_T(\mathbf{\check{x}}^l_T)-W^{u}_T(\mathbf{x}^l_T)=0$;
  \item[3)] For  $u'_T = l$ and $u_T\neq l$ it exists at least one channel $n$ such that $u'_T = n$ and $\check{x}^{(l)}_T {\geq}_s x^{(n)}_T {\geq}_s x^{(l)}_T$. It then holds that $0 \leq W^{u'}_T(\mathbf{\check{x}}^l_T)-W^{u}_T(\mathbf{x}^l_T)\leq R'(\check{x}^{(l)}_T-x^{(n)}_T)$.
\end{enumerate}
Therefore, Lemma~\ref{lemma:bound_p} holds for slot $T$.

Assume that Lemma~\ref{lemma:bound_p} holds for $T-1, \cdots, t+1$, then we prove the lemma for slot $t$.

\textbf{We first prove the first case: $u'_t=l$, $u_t=l$}.
By developing $\mathbf{\check{x}}^l_t$ and $\mathbf{\check{x}}^l_t$ according to Lemma~\ref{lemma:decomposability_sym_ARQ}, we have:
\begin{align}
\digamma(\mathbf{\check{x}}^l_t,u'_t)
=&\sum_{m\in \mathcal{Y}} d(\check{x}^{(l)}_t,m) \sum_{j\in\mathcal{X}} e'_j T(\check{x}^{(l)}_t,m)W^{\hat{u}'}_{t+1}(\mathbf{x}^{(-l)}_{t+1},e_j), \nonumber \\
=& \sum_{m\in \mathcal{Y}} \sum_{j\in\mathcal{X}} e'_j B(m)A'\check{x}^{(l)}_tW^{\hat{u}'}_{t+1}(\mathbf{x}^{(-l)}_{t+1},e_j) \\
\digamma(\mathbf{x}^l_t,u_t)
=& \sum_{m\in \mathcal{Y}} d(x^{(l)}_t,m) \sum_{j\in \mathcal{X}} e'_j T(x^{(l)}_t,m)W^{\hat{u}}_{t+1}(\mathbf{x}^{(-l)}_{t+1},e_j) \nonumber \\
=& \sum_{m\in \mathcal{Y}} \sum_{j\in\mathcal{X}} e'_j B(m)A' x^{(l)}_t W^{\hat{u}}_{t+1}(\mathbf{x}^{(-l)}_{t+1},e_j).
\end{align}
Furthermore, we have
\begin{align}
\label{eq:ih1_gamma_dif}
~&\digamma(\mathbf{\check{x}}^l_t,u'_t)-\digamma(\mathbf{x}^l_t,u_t) \nonumber \\
=& \sum_{m\in \mathcal{Y}} \sum_{j\in\mathcal{X}} \Big[e'_j B(m)A'\check{x}^{(l)}_tW^{\hat{u}'}_{t+1}(\mathbf{x}^{(-l)}_{t+1},e_j) \nonumber \\
&\qquad\qquad -e'_j B(m)A' x^{(l)}_t W^{\hat{u}}_{t+1}(\mathbf{x}^{(-l)}_{t+1},e_j) \Big] \nonumber \\
\overset{(a)}=&\sum_{m\in \mathcal{Y}}\sum_{j\in\mathcal{X}-\{1\}}\Big[e'_j B(m)A'(\check{x}^{(l)}_t-x^{(l)}_t) \nonumber \\
&\qquad\qquad\times\Big( W^{\hat{u}'}_{t+1}(\mathbf{x}^{(-l)}_{t+1},e_j)- W^{\hat{u}}_{t+1}(\mathbf{x}^{(-l)}_{t+1},e_1)\Big)\Big],
\end{align}
where, the equality (a) is due to $x^{(l)}_t(1)=1-\sum_{j\in\mathcal{X}^{(l)}-\{1\}}x^{(l)}_t(j)$.

Next, we analyze the term in the bracket, $W^{\hat{u}'}_{t+1}(\mathbf{x}^{(-l)}_{t+1},e_j)- W^{\hat{u}}_{t+1}(\mathbf{x}^{(-l)}_{t+1},e_1)$, of RHS  of~\eqref{eq:ih1_gamma_dif} through three cases:

Case 1: if $\hat{u}'_{t+1}=l$ and $\hat{u}_{t+1}=l$, according to the induction hypothesis, we have
\begin{align*}
  0 &\leq W^{\hat{u}'}_{t+1}(\mathbf{x}^{(-l)}_{t+1},e_j)- W^{\hat{u}}_{t+1}(\mathbf{x}^{(-l)}_{t+1},e_1)\\
    &\leq  \sum_{i=0}^{T-t-1} R' (\beta A')^i (e_j-e_1).
\end{align*}

Case 2: if $\hat{u}'_{t+1}\neq l$, $\hat{u}_{t+1}\neq l$, and $\hat{u}'_{t+1}=\hat{u}_{t+1}$, according to the induction hypothesis, we have
\begin{align*}
  0 &\leq W^{\hat{u}'}_{t+1}(\mathbf{x}^{(-l)}_{t+1},e_j)- W^{\hat{u}}_{t+1}(\mathbf{x}^{(-l)}_{t+1},e_1) \\
    &\leq  \sum_{i=1}^{T-t-1}R'(\beta A')^i (e_j-e_1) .
\end{align*}

Case 3: if  $\hat{u}'_{t+1}=l$ and $\hat{u}_{t+1}\neq l$, according to the induction hypothesis, we have
\begin{align*}
  0 &\leq W^{\hat{u}'}_{t+1}(\mathbf{x}^{(-l)}_{t+1},e_j)- W^{\hat{u}}_{t+1}(\mathbf{x}^{(-l)}_{t+1},e_1) \\
    &\leq  \sum_{i=0}^{T-t-1} R' (\beta A')^i(e_j-e_1) .
\end{align*}

Combining Case 1--3, we obtain the bounds of $W^{\hat{u}'}_{t+1}(\mathbf{x}^{(-l)}_{t+1},e_j)- W^{\hat{u}}_{t+1}(\mathbf{x}^{(-l)}_{t+1},e_1)$ as follows:
\begin{align*}
  0 &\leq W^{\hat{u}'}_{t+1}(\mathbf{x}^{(-l)}_{t+1},e_j)- W^{\hat{u}}_{t+1}(\mathbf{x}^{(-l)}_{t+1},e_1) \\
  &\leq  \sum_{i=0}^{T-t-1} R' (\beta A')^i(e_j-e_1).
\end{align*}

Therefore, we have
\begin{align*}
   & W^{u'}_t(\mathbf{\check{x}}^l_t)-W^{u}_t(\mathbf{x}^l_t) \\
  =  & R'(\check{x}^{(l)}_t-x^{(l)}_t)  +  \beta \digamma(\mathbf{\check{x}}^l_t,u'_t)-\digamma(\mathbf{x}^l_t,u_t)\\
 = & R'(\check{x}^{(l)}_t-x^{(l)}_t) +\beta  \sum_{m\in \mathcal{Y}}\sum_{j\in\mathcal{X}-\{1\}} \\
 & \Big[e'_j B(m)A'(\check{x}^{(l)}_t-x^{(l)}_t)\Big( W^{\hat{u}'}_{t+1}(\mathbf{x}^{(-l)}_{t+1},e_j)- W^{\hat{u}}_{t+1}(\mathbf{x}^{(-l)}_{t+1},e_1)\Big)\Big]\\
 \leq & R'(\check{x}^{(l)}_t-x^{(l)}_t)+ \beta  \sum_{m\in \mathcal{Y}}\sum_{j\in\mathcal{X}-\{1\}} \\
     &\Big[e'_j B(m)A'(\check{x}^{(l)}_t-x^{(l)}_t)\Big( \sum_{i=0}^{T-t-1} R' (\beta A')^i (e_j-e_1)\Big)\Big] \\
= & \sum_{i=0}^{T-t} R' (\beta A')^i (\check{x}^{(l)}_t-x^{(l)}_t).
\end{align*}

To the end, we complete the proof of the first part, $u'_t=l$ and $u_t=l$, of Lemma~\ref{lemma:bound_p}.

Secondly, \textbf{we prove the second case $u'_t\neq l$, $u_t\neq l$, and $u'_t=u_t$}, which implies that in this case, $u'_t=u_t$. Assuming $u'_t=u_t=k$, we have:
\begin{align}
&\digamma(\mathbf{\check{x}}^l_t,u'_t) \nonumber \\
=&\sum_{m\in \mathcal{Y}} d(x^{(k)}_t,m) \sum_{j\in\mathcal{X}} e'_j T(x^{(k)}_t,m)W^{\hat{u}'}_{t+1}(\mathbf{x}^{(-k,-l)}_{t+1},e_j,A'\check{x}^{(l)}_t) \nonumber \\
=& \sum_{m\in \mathcal{Y}} \sum_{j\in\mathcal{X}} e'_j B(m)A'x^{(k)}_t W^{\hat{u}'}_{t+1}(\mathbf{x}^{(-k,-l)}_{t+1},e_j,A'\check{x}^{(l)}_t) \\
&\digamma(\mathbf{x}^l_t,u_t) \nonumber  \\
=& \sum_{m\in \mathcal{Y}} d(x^{(k)}_t,m) \sum_{j\in \mathcal{X}} e'_j T(x^{(k)}_t,m)W^{\hat{u}}_{t+1}(\mathbf{x}^{(-k,-l)}_{t+1},e_j,A'x^{(l)}_t) \nonumber \\
=& \sum_{m\in \mathcal{Y}} \sum_{j\in\mathcal{X}} e'_j B(m)A' x^{(k)}_t W^{\hat{u}}_{t+1}(\mathbf{x}^{(-k,-l)}_{t+1},e_j,A' x^{(l)}_t).
\end{align}
Thus,
\begin{align}
\label{eq:case2}
~&\digamma(\mathbf{\check{x}}^l_t,u'_t)-\digamma(\mathbf{x}^l_t,u_t) \nonumber \\
=&\sum_{m\in \mathcal{Y}} \sum_{j\in\mathcal{X}} e'_j B(m)A' x^{(k)}_t  \nonumber \\ &\Big[W^{\hat{u}'}_{t+1}(\mathbf{x}^{(-k,-l)}_{t+1},e_j,A'\check{x}^{(l)}_t)-W^{\hat{u}}_{t+1}(\mathbf{x}^{(-k,-l)}_{t+1},e_j,A'x^{(l)}_t) \Big].
\end{align}

For the term in the bracket of RHS of~\eqref{eq:case2}, if $l$ is never chosen for $W^{\hat{u}'}_{t+1}(\mathbf{x}^{(-k,-l)}_{t+1},e_j,A'\check{x}^{(l)}_t)$ and $W^{\hat{u}}_{t+1}(\mathbf{x}^{(-k,-l)}_{t+1},e_j,A'x^{(l)}_t)$ from the slot $t+1$ to the end of time horizon of interest $T$. That is to say, $\hat{u}'_{\tau}\neq l$ and $\hat{u}_{\tau}\neq l$ for $t+1\le \tau \le T$, and further, we have $W^{\hat{u}'}_{t+1}(\mathbf{x}^{(-k,-l)}_{t+1},e_j,A'\check{x}^{(l)}_t)-W^{\hat{u}}_{t+1}(\mathbf{x}^{(-k,-l)}_{t+1},e_j,A'x^{(l)}_t) =0$; otherwise, it exists $t^o$ ($t+1\leq t^o \leq T$) such that one of the following three cases holds.

Case 1: $u'_{\tau}\neq l$ and $u_{\tau}\neq l$ for $t\le \tau \le t^0-1$ while $u'_{t^0} = l$ and $u_{t^0} = l$;

Case 2: $u'_{\tau}\neq l$ and $u_{\tau}\neq l$ for $t\le \tau \le t^0-1$ while $u'_{t^0} \neq l$ and $u_{t^0} = l$ (Note that this case does not exist since $R'[A']^{t^0-t}\check{x}^{(l)}_t \geq  R'[A']^{t^0-t}x^{(l)}_t$ according to the first order stochastic dominance of transition matrix $A$);

Case 3: $u'_{\tau}\neq l$ and $u_{\tau}\neq l$ for $t\le \tau \le t^0-1$ while $u'_{t^0} = l$ and $u_{t^0} \neq l$.

For Case 1, according to the hypothesis ($u'_{t^0} = l$ and $u_{t^0} = l$), we have
\begin{align*}
& \beta^{t_0-t}( W^{\hat{u}'}_{t^0}(\mathbf{\check{x}}^l_{t^0})-W^{\hat{u}}_{t^0}(\mathbf{x}^l_{t^0}) ) \\
&\le \beta^{t_0-t} \sum_{i=0}^{T-t^o}(\beta\overline{\lambda})^i R'(\check{x}^{(l)}_{t^0}- x^{(l)}_{t^0}) \\
&=\beta^{t_0-t}  \sum_{i=0}^{T-t^o} R'(\beta A')^i [A']^{t^0-t}(\check{x}^{(l)}_{t}- x^{(l)}_{t}) \\
&\overset{(b)}\leq \beta \sum_{i=0}^{T-t-1}R'(\beta A')^i A' (\check{x}^{(l)}_{t}- x^{(l)}_{t}),
\end{align*}
where, the inequality (b) is from $t^0\geq t+1$.

For Case 3, by the induction hypothesis, we have the similar results with Case 1.

Combing the results of the three cases, we obtain
\begin{align}
\label{eq:case22}
&W^{\hat{u}'}_{t+1}(\mathbf{x}^{(-k,-l)}_{t+1},e_j,A'\check{x}^{(l)}_t)-W^{\hat{u}}_{t+1}(\mathbf{x}^{(-k,-l)}_{t+1},e_j,A'x^{(l)}_t) \nonumber \\
&\leq \sum_{i=0}^{T-t-1} R'(\beta A')^i A' (\check{x}^{(l)}_{t}- x^{(l)}_{t}).
\end{align}

Combing~\eqref{eq:case22} and~\eqref{eq:case2}, we have
\begin{align*}
   & W^{u'}_t(\mathbf{\check{x}}^l_t)-W^{u}_t(\mathbf{x}^l_t) \\
   & = \beta (\digamma(\mathbf{\check{x}}^l_t,u'_t)-\digamma(\mathbf{x}^l_t,u_t)) \\
   &= \beta\sum_{m\in \mathcal{Y}} \sum_{j\in\mathcal{X}} e'_j B(m)A' x^{(k)}_t  \\ &\times\Big[W^{\hat{u}'}_{t+1}(\mathbf{x}^{(-k,-l)}_{t+1},e_j,A'\check{x}^{(l)}_t)-W^{\hat{u}}_{t+1}(\mathbf{x}^{(-k,-l)}_{t+1},e_j,A'x^{(l)}_t)\Big]\\
   &\leq \beta \sum_{m\in \mathcal{Y}} \sum_{j\in\mathcal{X}} e'_j B(m)A' x^{(k)}_t \sum_{i=0}^{T-t-1} R'(\beta A')^i A' (\check{x}^{(l)}_{t}- x^{(l)}_{t})\\
   &=  \sum_{j\in\mathcal{X}} e'_j\Big[ \sum_{m\in \mathcal{Y}}B(m)\Big]A' x^{(k)}_t \sum_{i=0}^{T-t-1} R'(\beta A')^{i+1} (\check{x}^{(l)}_{t}- x^{(l)}_{t})\\
   &=  \sum_{j\in\mathcal{X}} e'_j \mathbf{E} A' x^{(k)}_t \sum_{i=0}^{T-t-1} R'(\beta A')^{i+1} (\check{x}^{(l)}_{t}- x^{(l)}_{t})\\
   &=   \mathbf{1}'_{X}  A' x^{(k)}_t \sum_{i=0}^{T-t-1} R'(\beta A')^{i+1}(\check{x}^{(l)}_{t}- x^{(l)}_{t})\\
   &=   \mathbf{1}'_{X} x^{(k)}_t \sum_{i=0}^{T-t-1} R'(\beta A')^{i+1}(\check{x}^{(l)}_{t}- x^{(l)}_{t})\\
   &=  \sum_{i=0}^{T-t-1} R'(\beta A')^{i+1}(\check{x}^{(l)}_{t}- x^{(l)}_{t})\\
   &=  \sum_{i=1}^{T-t} R'(\beta A')^{i}(\check{x}^{(l)}_{t}- x^{(l)}_{t}),
\end{align*}
which completes the proof of Lemma~\ref{lemma:bound_p} when $l\notin \mathcal{A}'$ and $l\notin \mathcal{A}$.

Last, \textbf{we prove the third case $u'_t=l$ and $u_t\neq l$}, then it exists at least one process $u_t=n$, and its belief vector denoted as $x^{(n)}_t$, such that $\check{x}^{(l)}_t {\geq}_s x^{(n)}_t {\geq}_s x^{(l)}_t$. We have
\begin{align}
  &W^{u'}_t(\mathbf{\check{x}}^l_t)-W^{u}_t(\mathbf{x}^l_t) \nonumber \\
  =& W^{l}_t(x^{(1)}_t,\cdots,x^{(l-1)}_t,\check{x}^{(l)}_t,x^{(l+1)}_t,\cdots,x^{(N)}_t) \nonumber \\
  &- W^{n}_t(x^{(1)}_t,\cdots,x^{(l-1)}_t,x^{(l)}_t,x^{(l+1)}_t,\cdots,x^{(N)}_t)\nonumber \\
  =& [W^{l}_t(x^{(1)}_t,\cdots,x^{(l-1)}_t,\check{x}^{(l)}_t,x^{(l+1)}_t,\cdots,x^{(N)}_t) \nonumber \\
  &- W^{n}_t(x^{(1)}_t,\cdots,x^{(l-1)}_t,x^{(n)}_t,x^{(l+1)}_t,\cdots,x^{(N)}_t)] \nonumber\\
  +& [W^{n}_t(x^{(1)}_t,\cdots,x^{(l-1)}_t,x^{(n)}_t,x^{(l+1)}_t,\cdots,x^{(N)}_t) \nonumber \\
  &- W^{n}_t(x^{(1)}_t,\cdots,x^{(l-1)}_t,x^{(l)}_t,x^{(l+1)}_t,\cdots,x^{(N)}_t)] \nonumber\\
  =& [W^{l}_t(x^{(1)}_t,\cdots,x^{(l-1)}_t,\check{x}^{(l)}_t,x^{(l+1)}_t,\cdots,x^{(N)}_t) \nonumber \\
  &- W^{l}_t(x^{(1)}_t,\cdots,x^{(l-1)}_t,x^{(n)}_t,x^{(l+1)}_t,\cdots,x^{(N)}_t)] \nonumber \\
  +& [W^{n}_t(x^{(1)}_t,\cdots,x^{(l-1)}_t,x^{(n)}_t,x^{(l+1)}_t,\cdots,x^{(N)}_t) \nonumber \\
  &- W^{n}_t(x^{(1)}_t,\cdots,x^{(l-1)}_t,x^{(l)}_t,x^{(l+1)}_t,\cdots,x^{(N)}_t)].
\label{eq:case3}
\end{align}

According to the induction hypothesis ($l \in \mathcal{A}'$ and $l \in \mathcal{A}$), the first term of the RHS of~\eqref{eq:case3} can be bounded as follows:
\begin{align}
& W^{u'}_t(x^{(1)}_t,\cdots,x^{(l-1)}_t,\check{x}^{(l)}_t,x^{(l+1)}_t,\cdots,x^{(N)}_t) \nonumber \\
&- W^{u}_t(x^{(1)}_t,\cdots,x^{(l-1)}_t,x^{(m)}_t,x^{(l+1)}_t,\cdots,x^{(N)}_t) \nonumber \\
    &\leq \sum_{i=0}^{T-t}R' (\beta A')^i (\check{x}^{(l)}_t-x^{(n)}_t).
\label{eq:case3_1}
\end{align}

Meanwhile, the second term of the RHS of ~\eqref{eq:case3} is inducted by hypothesis ($l\notin \mathcal{A}'$ and $l\notin \mathcal{A}$):
\begin{align}
\label{eq:case3_2}
   &W^{u'}_t(x^{(1)}_t,\cdots,x^{(l-1)}_t,x^{(m)}_t,x^{(l+1)}_t,\cdots,x^{(N)}_t) \nonumber \\
    & - W^{u}_t(x^{(1)}_t,\cdots,x^{(l-1)}_t,x^{(l)}_t,x^{(l+1)}_t,\cdots,x^{(N)}_t) \nonumber \\
    &\leq  \sum_{i=1}^{T-t} R' (\beta A')^i(x^{(n)}_t-x^{(l)}_t).
\end{align}
Therefore, we have, combining~\eqref{eq:case3},~\eqref{eq:case3_1} and ~\eqref{eq:case3_2},
\begin{align*}
 W^{u'}_t(\mathbf{\check{x}}^l_t)-W^{u}_t(\mathbf{x}^l_t)  \leq  \sum_{i=0}^{T-t} R'(\beta A')^i(\check{x}^{(l)}_t-x^{(l)}_t).
 \end{align*}
Thus, we complete the proof of the third part, $l \in \mathcal{\mathcal{A}}'(t)$ and $l\notin \mathcal{A}(t)$, of Lemma~\ref{lemma:bound_p}.

To the end, Lemma~\ref{lemma:bound_p} is concluded.

\end{document}